\documentclass{amsart}
\usepackage{amsfonts}

\newtheorem{theorem}{Theorem}
\theoremstyle{plain}
\newtheorem{acknowledgement}{Acknowledgement}

\newtheorem{definition}{Definition}

\newtheorem{lemma}{Lemma}

\newtheorem{remark}{Remark}

\numberwithin{equation}{section}
\input{tcilatex}

\begin{document}
\title[Data Dependence of fixed points of...]{Data dependence results of a
new multistep and S-iterative schemes for contractive-like operators}
\author{Faik GURSOY}
\address{Department of Mathematics, Yildiz Technical University, Davutpasa
Campus, Esenler, 34220 Istanbul, Turkey}
\email{faikgursoy02@hotmail.com;fgursoy@yildiz.edu.tr}
\urladdr{http://www.yarbis.yildiz.edu.tr/fgursoy}
\author{Vatan KARAKAYA}
\curraddr{Department of Mathematical Engineering, Yildiz Technical
University, Davutpasa Campus, Esenler, 34210 Istanbul}
\email{vkkaya@yildiz.edu.tr;vkkaya@yahoo.com}
\urladdr{http://www.yarbis.yildiz.edu.tr/vkkaya}
\author{B. E. RHOADES}
\address{Department of Mathematics, Indiana University, Bloomington, IN
47405-7106, USA}
\email{rhoades@indiana.edu}
\urladdr{http://www.math.indiana.edu/people/profile.phtml?id=rhoades}
\date{December,2012}

\begin{abstract}
In this paper, we prove that convergence of a new iteration and S-iteration
can be used approximate the fixed points of contractive-like operators. We
also prove some data dependence results for these new iteration and
S-iteration schemes for contractive-like operators. Our results extend and
improve some known results in the literature.
\end{abstract}

\maketitle

\section{Introduction}

Contractive mappings and iteration procedures are some of the main tools in
the study of fixed point theory. There are many contractive mappings and
iteration schemes that have been introduced and developed by several authors
to serve various purposes in the literature of this highly active research
area, viz., \ \cite{Rhoades, Mann, Ishikawa, Noor, RS7, Thianwan, SP,
Glowinski, XuNoor, Takahashi, Das, Agarwall}, among others.

Whether an iteration method used in any investigation converges to a fixed
point of a contractive type mapping corresponding to a particular iteration
process is of utmost importance. Therefore it is natural to see many works
related to convergence of iteration methods, such as \cite{Berinde, Chidume,
Chidume1, Suantai, Sahu, Isa, CR, Osilike, Rafiq, Hussain}.

Fixed point theory is concerned with investigating a wide variety of issues
such as the existence (and uniqueness) of fixed points, the construction of
fixed points, etc. One of these themes is data dependency of fixed points.
Data dependency of fixed points has been the subject of research in fixed
point theory for some time now, and data dependence research is an important
theme in its own right.

Several authors who have made contributions to the study of data dependence
of fixed points are Rus and Muresan \cite{Rus1}, Rus et al. \cite{Rus2, Rus3}%
, Berinde \cite{VBerinde}, Esp\'{\i}nola and Petru\c{s}el \cite{Espinola},
Markin \cite{Markin}, Chifu and Petru\c{s}el \cite{Chifu}, Olantiwo \cite%
{Olantiwo, Olantiwo1}, \c{S}oltuz \cite{DataM, Data Is 1}, \c{S}oltuz and
Grosan \cite{Data Is 2}, Chugh and Kumar \cite{DataSP} and the references
therein.

This paper is organized as follows. In Section 1 we present a brief survey
of some known contractive mappings and iterative schemes and collect some
preliminaries that will be used in the proofs of our main results. In
Section 2 we show that the convergence of a new multi-step iteration, which
is a special case of Jungck multistep-SP iterative process defined in \cite%
{Akewe} and S-iteration(due to Agarwal et al.) can be used approximate the
fixed points of contractive-like operators. Motivated by the works of \c{S}%
oltuz \cite{DataM, Data Is 1}, \c{S}oltuz and Grosan \cite{Data Is 2}, and
Chugh and Kumar \cite{DataSP}, we prove two data dependence results for the
new multi-step iteration and S-iteration schemes by employing
contractive-like operators.

As a background of our exposition, we now mention some contractive mappings
and iteration schemes.

In \cite{Zamfirescu} Zamfirescu established an important generalization of
the Banach fixed point theorem using the following contractive condition:
For a mapping $T:E\rightarrow E$, there exist real numbers $a$, $b$, $c$
satisfying $0<a<1$, $0<b$, $c<1/2$ such that, for each pair $x$, $y\in X$,
at least one of the following is true:%
\begin{equation}
\left\{ 
\begin{array}{c}
\text{(z}_{\text{1}}\text{) \ \ \ \ \ \ \ \ \ \ \ \ \ \ \ \ \ \ \ \ \ \ \ \ }%
\left\Vert Tx-Ty\right\Vert \leq a\left\Vert x-y\right\Vert \text{,} \\ 
\text{(z}_{\text{2}}\text{) \ \ \ \ }\left\Vert Tx-Ty\right\Vert \leq
b\left( \left\Vert x-Tx\right\Vert +\left\Vert y-Ty\right\Vert \right) \text{%
,} \\ 
\text{(z}_{\text{3}}\text{) \ \ \ \ }\left\Vert Tx-Ty\right\Vert \leq
c\left( \left\Vert x-Ty\right\Vert +\left\Vert y-Tx\right\Vert \right) \text{%
.}%
\end{array}%
\right.  \label{eqn1}
\end{equation}%
A mapping $T$ satisfying the contractive conditions (z$_{\text{1}}$), (z$_{%
\text{2}}$) and (z$_{\text{3}}$) in (1.1) is called a Zamfirescu operator.
An operator satisfying condition (z$_{\text{2}}$) is called a \textit{Kannan
operator}, while the mapping satisfying condition (z$_{\text{3}}$) is called
a \textit{Chatterjea operator}. As shown in \cite{Berinde}, the contractive
condition (1.1) leads to%
\begin{equation}
\left\{ 
\begin{array}{c}
\text{(b}_{\text{1}}\text{) \ \ \ \ }\left\Vert Tx-Ty\right\Vert \leq \delta
\left\Vert x-y\right\Vert +2\delta \left\Vert x-Tx\right\Vert \text{ if one
use (z}_{\text{2}}\text{),} \\ 
\text{and \ \ \ \ \ \ \ \ \ \ \ \ \ \ \ \ \ \ \ \ \ \ \ \ \ \ \ \ \ \ \ \ \
\ \ \ \ \ \ \ \ \ \ \ \ \ \ \ \ \ \ \ \ \ \ \ \ \ \ \ \ \ \ \ \ \ \ \ \ \ \
\ \ \ \ \ \ \ \ \ \ \ \ \ \ } \\ 
\text{(b}_{\text{2}}\text{) \ \ \ \ }\left\Vert Tx-Ty\right\Vert \leq \delta
\left\Vert x-y\right\Vert +2\delta \left\Vert x-Ty\right\Vert \text{ if one
use (z}_{\text{3}}\text{),}%
\end{array}%
\right.  \label{eqn2}
\end{equation}%
for all $x$, $y\in E$ where $\delta :=\max \left\{ a,\frac{b}{1-b},\frac{c}{%
1-c}\right\} $, $\delta \in \left[ 0,1\right) $, and it was shown that this
class of operators is wider than the class of Zamfirescu operators. Any
mapping satisfying condition (b$_{\text{1}}$) or (b$_{\text{2}}$) is called
a quasi-contractive operator.

Extending the above definition, Osilike and Udomene \cite{Osilike}
considered operators $T$ for which there exist real numbers $L\geq 0$ and $%
\delta \in \left[ 0,1\right) $ such that for all $x$, $y\in E$, 
\begin{equation}
\left\Vert Tx-Ty\right\Vert \leq \delta \left\Vert x-y\right\Vert
+L\left\Vert x-Tx\right\Vert \text{.}  \label{eqn3}
\end{equation}%
Imoru and Olantiwo \cite{Imoru} gave a more general definition: The operator 
$T$\ is called a contractive-like operator if there exists a constant $%
\delta \in \left[ 0,1\right) $\ and a strictly increasing and continuous
function $\varphi :\left[ 0,\infty \right) \rightarrow \left[ 0,\infty
\right) $\ with $\varphi \left( 0\right) =0$,\ such that, for each $x,y\in E$%
,%
\begin{equation}
\left\Vert Tx-Ty\right\Vert \leq \delta \left\Vert x-y\right\Vert +\varphi
\left( \left\Vert x-Tx\right\Vert \right) \text{.}  \label{eqn4}
\end{equation}%
A map satisfying (1.4) need not have a fixed point, even if $E$ is complete.
For example, let $E=\left[ 0,\infty \right) $, and define $T$ by%
\begin{equation*}
Tx=\left\{ 
\begin{array}{c}
1.0\text{,\ \ \ \ \ \ }0\leq x\leq 0.8\text{, \ \ \ \ } \\ 
0.6\text{, \ \ \ }0.8<x\text{. \ \ \ \ \ \ \ \ \ \ \ \ }%
\end{array}%
\right.
\end{equation*}%
Without loss of generality we may assume that $x<y$. Then, for $0\leq
x<y\leq 0.8$ or $0.8<x<y$, $\left\Vert Tx-Ty\right\Vert =0$, and (1.4)
automatically satisfied.

If $0\leq x\leq 0.8<y$, then $\left\Vert Tx-Ty\right\Vert =0.4$.

Define $\varphi $ by $\varphi \left( t\right) =Lt$ for any $L\geq 2$. Then $%
\varphi $ is increasing, continuous, and $\varphi \left( 0\right) =0$. Also, 
$\left\Vert x-Tx\right\Vert =1-x$, so that $\varphi \left( \left\Vert
x-Tx\right\Vert \right) =L\left( 1-x\right) \geq 0.2L\geq 0.4$.

Therefore%
\begin{equation*}
0.4=\left\Vert Tx-Ty\right\Vert \leq L\left\Vert x-Tx\right\Vert \leq \delta
\left\Vert x-y\right\Vert +L\left\Vert x-Tx\right\Vert
\end{equation*}%
for any $0\leq \delta <1$, and (1.4) is satisfied for $0\leq x\leq 0.8<y$.
But $T$ has no fixed point.

However, using (1.4) it is obvious that, if $T$ has a fixed point, then it
is unique.

Throughout this paper $%
\mathbb{N}
$ denotes the set of all nonnegative integers. Let $X$ be a Banach space, $%
E\subset X$ a nonempty closed, convex subset of $X$, and $T$ a self map on $%
E $. Define $F_{T}:=\left\{ p\in X:~p=Tp\right\} $ to be the set of fixed
points of $T$. Let $\left\{ \alpha _{n}\right\} _{n=0}^{\infty }$, $\left\{
\beta _{n}\right\} _{n=0}^{\infty }$, $\left\{ \gamma _{n}\right\}
_{n=0}^{\infty }$ and $\left\{ \beta _{n}^{i}\right\} _{n=0}^{\infty }$, $i=%
\overline{1,k-2}$, $k\geq 2$ be real sequences in $\left[ 0,1\right) $
satisfying certain conditions.

In \cite{RS7} Rhoades and \c{S}oltuz introduced a multi-step iterative
procedure defined by%
\begin{equation}
\left\{ 
\begin{array}{c}
x_{0}\in E\text{, \ \ \ \ \ \ \ \ \ \ \ \ \ \ \ \ \ \ \ \ \ \ \ \ \ \ \ \ \
\ \ \ \ \ \ \ \ \ \ \ \ \ \ \ \ \ \ \ \ \ } \\ 
y_{n}^{k-1}=\left( 1-\beta _{n}^{k-1}\right) x_{n}+\beta
_{n}^{k-1}Tx_{n},~k\geq 2\text{, \ \ \ \ \ } \\ 
y_{n}^{i}=\left( 1-\beta _{n}^{i}\right) x_{n}+\beta _{n}^{i}Ty_{n}^{i+1}%
\text{,}~i=\overline{1,k-2}\text{,} \\ 
x_{n+1}=\left( 1-\alpha _{n}\right) x_{n}+\alpha _{n}Ty_{n}^{1}\text{, \ }%
n\in 
\mathbb{N}
\text{. \ \ \ \ \ \ \ \ \ \ \ }%
\end{array}%
\right.  \label{eqn5}
\end{equation}%
The sequence $\left\{ x_{n}\right\} _{n=0}^{\infty }$ defined by%
\begin{equation}
\left\{ 
\begin{array}{c}
x_{0}\in E\text{, \ \ \ \ \ \ \ \ \ \ \ \ \ \ \ \ \ \ \ \ \ \ \ \ \ \ \ \ \
\ \ \ \ \ \ \ \ \ \ \ \ } \\ 
x_{n+1}=\left( 1-\alpha _{n}\right) Tx_{n}+\alpha _{n}Ty_{n}\text{, \ \ \ \
\ \ \ \ \ \ \ \ } \\ 
y_{n}=\left( 1-\beta _{n}\right) x_{n}+\beta _{n}Tx_{n}\text{,}~n\in 
\mathbb{N}
\text{, \ }%
\end{array}%
\right.  \label{eqn6}
\end{equation}%
is known as the S-iteration process (see \cite{Sahu, Agarwal, Agarwall}).

S.Thianwan \cite{Thianwan} defined a two-step iteration $\left\{
u_{n}\right\} _{n=0}^{\infty }$ by 
\begin{equation}
\left\{ 
\begin{array}{c}
x_{0}\in E\text{, \ \ \ \ \ \ \ \ \ \ \ \ \ \ \ \ \ \ \ \ \ \ \ \ \ \ \ \ \
\ \ \ \ \ \ \ \ \ \ \ \ } \\ 
x_{n+1}=\left( 1-\alpha _{n}\right) y_{n}+\alpha _{n}Ty_{n}\text{, \ \ \ \ \
\ \ \ \ \ \ \ } \\ 
y_{n}=\left( 1-\beta _{n}\right) x_{n}+\beta _{n}Tx_{n}\text{,}~n\in 
\mathbb{N}
\text{.}%
\end{array}%
\right.  \label{eqn7}
\end{equation}%
Recently Phuengrattana and Suantai \cite{SP} introduced an SP iteration
method defined by%
\begin{equation}
\left\{ 
\begin{array}{c}
x_{0}\in E\text{, \ \ \ \ \ \ \ \ \ \ \ \ \ \ \ \ \ \ \ \ \ \ \ \ \ \ \ \ \
\ \ \ \ \ \ } \\ 
x_{n+1}=\left( 1-\alpha _{n}\right) y_{n}+\alpha _{n}Ty_{n}\text{, \ \ \ \ \
\ \ \ \ \ \ \ \ } \\ 
y_{n}=\left( 1-\beta _{n}\right) z_{n}+\beta _{n}Tz_{n}\text{, \ \ \ \ \ \ \
\ \ \ } \\ 
z_{n}=\left( 1-\gamma _{n}\right) x_{n}+\gamma _{n}Tx_{n},~n\in 
\mathbb{N}
\text{.}%
\end{array}%
\right.  \label{eqn8}
\end{equation}%
We shall employ the following iterative process. For an arbitrary fixed
order $k\geq 2$,

\begin{equation}
\left\{ 
\begin{array}{c}
x_{0}\in E\text{, \ \ \ \ \ \ \ \ \ \ \ \ \ \ \ \ \ \ \ \ \ \ \ \ \ \ \ \ \
\ \ \ \ \ \ \ \ \ \ \ \ \ \ \ \ } \\ 
x_{n+1}=\left( 1-\alpha _{n}\right) y_{n}^{1}+\alpha _{n}Ty_{n}^{1}\text{, \
\ \ \ \ \ \ \ \ \ \ \ \ \ \ \ \ \ } \\ 
y_{n}^{1}=\left( 1-\beta _{n}^{1}\right) y_{n}^{2}+\beta _{n}^{1}Ty_{n}^{2}%
\text{, \ \ \ \ \ \ \ \ \ \ \ \ \ \ } \\ 
y_{n}^{2}=\left( 1-\beta _{n}^{2}\right) y_{n}^{3}+\beta _{n}^{2}Ty_{n}^{3}%
\text{, \ \ \ \ \ \ \ \ \ \ \ \ \ \ } \\ 
\cdots \text{ \ \ \ \ \ \ \ \ \ \ \ \ \ \ \ \ \ \ \ \ \ \ \ \ \ \ \ \ \ \ \
\ \ \ \ \ \ \ \ \ } \\ 
y_{n}^{k-2}=\left( 1-\beta _{n}^{k-2}\right) y_{n}^{k-1}+\beta
_{n}^{k-2}Ty_{n}^{k-1}\text{, \ \ \ \ } \\ 
y_{n}^{k-1}=\left( 1-\beta _{n}^{k-1}\right) x_{n}+\beta
_{n}^{k-1}Tx_{n},~n\in 
\mathbb{N}
\text{,}%
\end{array}%
\right.  \label{eqn9}
\end{equation}%
or, in short,%
\begin{equation}
\left\{ 
\begin{array}{c}
x_{0}\in E\text{, \ \ \ \ \ \ \ \ \ \ \ \ \ \ \ \ \ \ \ \ \ \ \ \ \ \ \ \ \
\ \ \ \ \ \ \ \ \ \ \ \ \ \ \ \ \ \ \ \ \ \ } \\ 
x_{n+1}=\left( 1-\alpha _{n}\right) y_{n}^{1}+\alpha _{n}Ty_{n}^{1}\text{, \
\ \ \ \ \ \ \ \ \ \ \ \ \ \ \ \ \ \ \ \ \ \ \ \ } \\ 
y_{n}^{i}=\left( 1-\beta _{n}^{i}\right) y_{n}^{i+1}+\beta
_{n}^{i}Ty_{n}^{i+1}\text{,}~~i=\overline{1,k-2}\text{, } \\ 
y_{n}^{k-1}=\left( 1-\beta _{n}^{k-1}\right) x_{n}+\beta
_{n}^{k-1}Tx_{n},~n\in 
\mathbb{N}
\text{, \ \ \ \ \ \ \ }%
\end{array}%
\right.  \label{eqn10}
\end{equation}%
where 
\begin{equation}
\left\{ \alpha _{n}\right\} _{n=0}^{\infty }\subset \left[ 0,1\right) \text{%
, }\sum\limits_{n=0}^{\infty }\alpha _{n}=\infty \text{,}  \label{eqn11}
\end{equation}%
and%
\begin{equation}
\left\{ \beta _{n}^{i}\right\} _{n=0}^{\infty }\subset \left[ 0,1\right) 
\text{, }i=\overline{1,k-1}\text{.}  \label{eqn12}
\end{equation}

\begin{remark}
If $\gamma _{n}=0$, then SP iteration $(1.8)$ reduces to the two-step
iteration $(1.7)$. By taking $k=3$ and $k=2$ in $(1.10)$ we obtain the
iterations $(1.8)$ and $(1.7)$, respectively.
\end{remark}

We shall need following definition and lemma in the sequel.

\begin{definition}
\cite{Vasile} Let $T$,$\widetilde{T}:X\rightarrow X$ be two operators. We
say that $\widetilde{T}$ is an approximate operator for $T$ if, for some $%
\varepsilon >0$, we have%
\begin{equation*}
\left\Vert Tx-\widetilde{T}x\right\Vert \leq \varepsilon \text{,}
\end{equation*}%
for all $x\in X$.
\end{definition}

\begin{lemma}
\cite{Data Is 2} Let $\left\{ a_{n}\right\} _{n=0}^{\infty }$ be a
nonnegative sequence for which one assumes that there exists an $n_{0}\in 
\mathbb{N}
$, such that, for all $n\geq n_{0}$,%
\begin{equation*}
a_{n+1}\leq \left( 1-\mu _{n}\right) a_{n}+\mu _{n}\eta _{n}
\end{equation*}%
is satisfied, where $\mu _{n}\in \left( 0,1\right) $, for all $n\in 
\mathbb{N}
$, $\sum\limits_{n=0}^{\infty }\mu _{n}=\infty $ and $\eta _{n}\geq 0$, $%
\forall n\in 
\mathbb{N}
$. Then the following holds:%
\begin{equation*}
0\leq \lim_{n\rightarrow \infty }\sup a_{n}\leq \lim_{n\rightarrow \infty
}\sup \eta _{n}\text{.}
\end{equation*}
\end{lemma}

\section{Main Results}

For simplicity we use the following notation through this section.

For any iterative process, $\left\{ x_{n}\right\} _{n=0}^{\infty }$ and $%
\left\{ u_{n}\right\} _{n=0}^{\infty }$ denote iterative sequences
associated to $T$ and $\widetilde{T}$, respectively.

\begin{theorem}
Let $T:E\rightarrow E$ be a map satisfying $\left( 1.4\right) $ with $%
F_{T}\neq \emptyset $ and $\left\{ x_{n}\right\} _{n=0}^{\infty }$ is a
sequence defined by $(1.10)$, then the sequence $\left\{ x_{n}\right\}
_{n=0}^{\infty }$ converges to the unique fixed point of $T$.
\end{theorem}

\begin{proof}
The proof can be easily obtained by using argument in the proof of (\cite%
{Akewe}, Theorem 3.1).
\end{proof}

This result allow us to prove the following theorem.

\begin{theorem}
Let $T:E\rightarrow E$ be a map satisfying $(1.4)$ with $F_{T}\neq \emptyset 
$ and $\widetilde{T}$ be an approximate operator of $T$ as in the Definition
1. Let $\left\{ x_{n}\right\} _{n=0}^{\infty }$, $\left\{ u_{n}\right\}
_{n=0}^{\infty }$ be two iterative sequences defined by $(1.10)$ and with
real sequences $\left\{ \alpha _{n}\right\} _{n=0}^{\infty }$, $\left\{
\beta _{n}^{i}\right\} _{n=0}^{\infty }$ $\subset \left[ 0,1\right) $
satisfying (i) $0\leq \beta _{n}^{i}<\alpha _{n}\leq 1$, $i=\overline{1,k-1}$%
, (ii) $\sum \alpha _{n}=\infty $. If $p=Tp$\ and $q=\widetilde{T}q$, then
we have%
\begin{equation*}
\left\Vert p-q\right\Vert \leq \frac{k\varepsilon }{1-\delta }\text{.}
\end{equation*}
\end{theorem}

\begin{proof}
For a given $x_{0}\in E$ and $u_{0}\in E$ we consider following multistep
iteration for $T$ and $\widetilde{T}$:%
\begin{equation}
\left\{ 
\begin{array}{c}
x_{0}\in E\text{, \ \ \ \ \ \ \ \ \ \ \ \ \ \ \ \ \ \ \ \ \ \ \ \ \ \ \ \ \
\ \ \ \ \ \ \ \ \ \ \ \ \ \ \ \ \ \ \ \ \ \ } \\ 
x_{n+1}=\left( 1-\alpha _{n}\right) y_{n}^{1}+\alpha _{n}Ty_{n}^{1}\text{, \
\ \ \ \ \ \ \ \ \ \ \ \ \ \ \ \ \ \ \ \ \ \ \ \ } \\ 
y_{n}^{i}=\left( 1-\beta _{n}^{i}\right) y_{n}^{i+1}+\beta
_{n}^{i}Ty_{n}^{i+1}\text{,}~~i=\overline{1,k-2}\text{, } \\ 
y_{n}^{k-1}=\left( 1-\beta _{n}^{k-1}\right) x_{n}+\beta _{n}^{k-1}Tx_{n}%
\text{, }k\geq 2\text{, }n\in 
\mathbb{N}
\text{, \ \ \ \ \ \ \ }%
\end{array}%
\right.  \label{eqn25}
\end{equation}%
and%
\begin{equation}
\left\{ 
\begin{array}{c}
u_{0}\in E\text{, \ \ \ \ \ \ \ \ \ \ \ \ \ \ \ \ \ \ \ \ \ \ \ \ \ \ \ \ \
\ \ \ \ \ \ \ \ \ \ \ \ \ \ \ \ \ \ \ \ \ \ } \\ 
u_{n+1}=\left( 1-\alpha _{n}\right) v_{n}^{1}+\alpha _{n}\widetilde{T}%
v_{n}^{1}\text{, \ \ \ \ \ \ \ \ \ \ \ \ \ \ \ \ \ \ \ \ \ \ \ \ \ } \\ 
v_{n}^{i}=\left( 1-\beta _{n}^{i}\right) v_{n}^{i+1}+\beta _{n}^{i}%
\widetilde{T}v_{n}^{i+1}\text{,}~~i=\overline{1,k-2}\text{, } \\ 
v_{n}^{k-1}=\left( 1-\beta _{n}^{k-1}\right) u_{n}+\beta _{n}^{k-1}%
\widetilde{T}u_{n}\text{, }k\geq 2\text{, }n\in 
\mathbb{N}
\text{. \ \ \ \ \ \ \ }%
\end{array}%
\right.  \label{eqn26}
\end{equation}%
Then, from (1.4), (2.1) and (2.2), we have the following estimates.%
\begin{eqnarray}
\left\Vert x_{n+1}-u_{n+1}\right\Vert &=&\left\Vert \left( 1-\alpha
_{n}\right) \left( y_{n}^{1}-v_{n}^{1}\right) +\alpha _{n}\left( Ty_{n}^{1}-%
\widetilde{T}v_{n}^{1}\right) \right\Vert  \notag \\
&\leq &\left( 1-\alpha _{n}\right) \left\Vert y_{n}^{1}-v_{n}^{1}\right\Vert
+\alpha _{n}\left\Vert Ty_{n}^{1}-\widetilde{T}v_{n}^{1}\right\Vert  \notag
\\
&=&\left( 1-\alpha _{n}\right) \left\Vert y_{n}^{1}-v_{n}^{1}\right\Vert
+\alpha _{n}\left\Vert Ty_{n}^{1}-Tv_{n}^{1}+Tv_{n}^{1}-\widetilde{T}%
v_{n}^{1}\right\Vert  \notag \\
&\leq &\left( 1-\alpha _{n}\right) \left\Vert y_{n}^{1}-v_{n}^{1}\right\Vert
+\alpha _{n}\left\Vert Ty_{n}^{1}-Tv_{n}^{1}\right\Vert +\alpha
_{n}\left\Vert Tv_{n}^{1}-\widetilde{T}v_{n}^{1}\right\Vert  \notag \\
&\leq &\left( 1-\alpha _{n}\right) \left\Vert y_{n}^{1}-v_{n}^{1}\right\Vert
+\alpha _{n}\delta \left\Vert y_{n}^{1}-v_{n}^{1}\right\Vert +\alpha
_{n}\varphi \left( \left\Vert y_{n}^{1}-Ty_{n}^{1}\right\Vert \right)
+\alpha _{n}\varepsilon  \notag \\
&=&\left[ 1-\alpha _{n}\left( 1-\delta \right) \right] \left\Vert
y_{n}^{1}-v_{n}^{1}\right\Vert +\alpha _{n}\varphi \left( \left\Vert
y_{n}^{1}-Ty_{n}^{1}\right\Vert \right) +\alpha _{n}\varepsilon \text{,}
\label{eqn27}
\end{eqnarray}%
\begin{eqnarray}
\left\Vert y_{n}^{1}-v_{n}^{1}\right\Vert &=&\left\Vert \left( 1-\beta
_{n}^{1}\right) \left( y_{n}^{2}-v_{n}^{2}\right) +\beta _{n}^{1}\left(
Ty_{n}^{2}-\widetilde{T}v_{n}^{2}\right) \right\Vert  \notag \\
&\leq &\left( 1-\beta _{n}^{1}\right) \left\Vert
y_{n}^{2}-v_{n}^{2}\right\Vert +\beta _{n}^{1}\left\Vert Ty_{n}^{2}-%
\widetilde{T}v_{n}^{2}\right\Vert  \notag \\
&\leq &\left( 1-\beta _{n}^{1}\right) \left\Vert
y_{n}^{2}-v_{n}^{2}\right\Vert +\beta _{n}^{1}\left\Vert
Ty_{n}^{2}-Tv_{n}^{2}\right\Vert +\beta _{n}^{1}\left\Vert Tv_{n}^{2}-%
\widetilde{T}v_{n}^{2}\right\Vert  \notag \\
&\leq &\left( 1-\beta _{n}^{1}\right) \left\Vert
y_{n}^{2}-v_{n}^{2}\right\Vert +\beta _{n}^{1}\delta \left\Vert
y_{n}^{2}-v_{n}^{2}\right\Vert +\beta _{n}^{1}\varphi \left( \left\Vert
y_{n}^{2}-Ty_{n}^{2}\right\Vert \right) +\beta _{n}^{1}\varepsilon  \notag \\
&=&\left[ 1-\beta _{n}^{1}\left( 1-\delta \right) \right] \left\Vert
y_{n}^{2}-v_{n}^{2}\right\Vert +\beta _{n}^{1}\varphi \left( \left\Vert
y_{n}^{2}-Ty_{n}^{2}\right\Vert \right) +\beta _{n}^{1}\varepsilon \text{,}
\label{eqn28}
\end{eqnarray}%
\begin{eqnarray}
\left\Vert y_{n}^{2}-v_{n}^{2}\right\Vert &=&\left\Vert \left( 1-\beta
_{n}^{2}\right) \left( y_{n}^{3}-v_{n}^{3}\right) +\beta _{n}^{2}\left(
Ty_{n}^{3}-\widetilde{T}v_{n}^{3}\right) \right\Vert  \notag \\
&\leq &\left( 1-\beta _{n}^{2}\right) \left\Vert
y_{n}^{3}-v_{n}^{3}\right\Vert +\beta _{n}^{2}\left\Vert Ty_{n}^{3}-%
\widetilde{T}v_{n}^{3}\right\Vert  \notag \\
&\leq &\left( 1-\beta _{n}^{2}\right) \left\Vert
y_{n}^{3}-v_{n}^{3}\right\Vert +\beta _{n}^{2}\left\Vert
Ty_{n}^{3}-Tv_{n}^{3}\right\Vert +\beta _{n}^{2}\left\Vert Tv_{n}^{3}-%
\widetilde{T}v_{n}^{3}\right\Vert  \notag \\
&\leq &\left( 1-\beta _{n}^{2}\right) \left\Vert
y_{n}^{3}-v_{n}^{3}\right\Vert +\beta _{n}^{2}\delta \left\Vert
y_{n}^{3}-v_{n}^{3}\right\Vert +\beta _{n}^{2}\varphi \left( \left\Vert
y_{n}^{3}-Ty_{n}^{3}\right\Vert \right) +\beta _{n}^{2}\varepsilon  \notag \\
&=&\left[ 1-\beta _{n}^{2}\left( 1-\delta \right) \right] \left\Vert
y_{n}^{3}-v_{n}^{3}\right\Vert +\beta _{n}^{2}\varphi \left( \left\Vert
y_{n}^{3}-Ty_{n}^{3}\right\Vert \right) +\beta _{n}^{2}\varepsilon \text{.}
\label{eqn29}
\end{eqnarray}%
Combining (2.3), (2.4) and (2.5) we obtain%
\begin{eqnarray}
\left\Vert x_{n+1}-u_{n+1}\right\Vert &\leq &\left[ 1-\alpha _{n}\left(
1-\delta \right) \right] \left[ 1-\beta _{n}^{1}\left( 1-\delta \right) %
\right] \left[ 1-\beta _{n}^{2}\left( 1-\delta \right) \right] \left\Vert
y_{n}^{3}-v_{n}^{3}\right\Vert  \notag \\
&&+\left[ 1-\alpha _{n}\left( 1-\delta \right) \right] \left[ 1-\beta
_{n}^{1}\left( 1-\delta \right) \right] \beta _{n}^{2}\varphi \left(
\left\Vert y_{n}^{3}-Ty_{n}^{3}\right\Vert \right)  \notag \\
&&+\left[ 1-\alpha _{n}\left( 1-\delta \right) \right] \left[ 1-\beta
_{n}^{1}\left( 1-\delta \right) \right] \beta _{n}^{2}\varepsilon  \notag \\
&&+\left[ 1-\alpha _{n}\left( 1-\delta \right) \right] \beta _{n}^{1}\varphi
\left( \left\Vert y_{n}^{2}-Ty_{n}^{2}\right\Vert \right)  \notag \\
&&+\left[ 1-\alpha _{n}\left( 1-\delta \right) \right] \beta
_{n}^{1}\varepsilon +\alpha _{n}\varphi \left( \left\Vert
y_{n}^{1}-Ty_{n}^{1}\right\Vert \right) +\alpha _{n}\varepsilon \text{.}
\label{eqn30}
\end{eqnarray}%
Thus, by induction, we get%
\begin{eqnarray}
\left\Vert x_{n+1}-u_{n+1}\right\Vert &\leq &\left[ 1-\alpha _{n}\left(
1-\delta \right) \right]  \notag \\
&&\left[ 1-\beta _{n}^{1}\left( 1-\delta \right) \right] \cdots \left[
1-\beta _{n}^{k-2}\left( 1-\delta \right) \right] \left\Vert
y_{n}^{k-1}-v_{n}^{k-1}\right\Vert  \notag \\
&&+\left[ 1-\alpha _{n}\left( 1-\delta \right) \right]  \notag \\
&&\left[ 1-\beta _{n}^{1}\left( 1-\delta \right) \right] \cdots \left[
1-\beta _{n}^{k-3}\left( 1-\delta \right) \right] \beta _{n}^{k-2}\varphi
\left( \left\Vert y_{n}^{k-1}-Ty_{n}^{k-1}\right\Vert \right)  \notag \\
&&+\cdots +\left[ 1-\alpha _{n}\left( 1-\delta \right) \right] \beta
_{n}^{1}\varphi \left( \left\Vert y_{n}^{2}-Ty_{n}^{2}\right\Vert \right)
+\alpha _{n}\varphi \left( \left\Vert y_{n}^{1}-Ty_{n}^{1}\right\Vert \right)
\notag \\
&&+\left[ 1-\alpha _{n}\left( 1-\delta \right) \right] \left[ 1-\beta
_{n}^{1}\left( 1-\delta \right) \right] \cdots \left[ 1-\beta
_{n}^{k-3}\left( 1-\delta \right) \right] \beta _{n}^{k-2}\varepsilon  \notag
\\
&&+\cdots +\left[ 1-\alpha _{n}\left( 1-\delta \right) \right] \beta
_{n}^{1}\varepsilon +\alpha _{n}\varepsilon \text{.}  \label{eqn31}
\end{eqnarray}%
Again using (1.4), (2.1) and (2.2) we get%
\begin{eqnarray}
\left\Vert y_{n}^{k-1}-v_{n}^{k-1}\right\Vert &=&\left\Vert \left( 1-\beta
_{n}^{k-1}\right) \left( x_{n}-u_{n}\right) +\beta _{n}^{k-1}\left( Tx_{n}-%
\widetilde{T}u_{n}\right) \right\Vert  \notag \\
&\leq &\left( 1-\beta _{n}^{k-1}\right) \left\Vert x_{n}-u_{n}\right\Vert
+\beta _{n}^{k-1}\left\Vert Tx_{n}-\widetilde{T}u_{n}\right\Vert  \notag \\
&\leq &\left( 1-\beta _{n}^{k-1}\right) \left\Vert x_{n}-u_{n}\right\Vert
+\beta _{n}^{k-1}\left\Vert Tx_{n}-Tu_{n}\right\Vert  \notag \\
&&+\beta _{n}^{k-1}\left\Vert Tu_{n}-\widetilde{T}u_{n}\right\Vert  \notag \\
&\leq &\left[ 1-\beta _{n}^{k-1}\left( 1-\delta \right) \right] \left\Vert
x_{n}-u_{n}\right\Vert +\beta _{n}^{k-1}\varphi \left( \left\Vert
x_{n}-Tx_{n}\right\Vert \right) +\beta _{n}^{k-1}\varepsilon \text{.}
\label{eqn32}
\end{eqnarray}%
Substituting (2.8) in (2.7) we have%
\begin{eqnarray}
\left\Vert x_{n+1}-u_{n+1}\right\Vert &\leq &\left[ 1-\alpha _{n}\left(
1-\delta \right) \right]  \notag \\
&&\left[ 1-\beta _{n}^{1}\left( 1-\delta \right) \right] \cdots \left[
1-\beta _{n}^{k-1}\left( 1-\delta \right) \right] \left\Vert
x_{n}-u_{n}\right\Vert  \notag \\
&&+\left[ 1-\alpha _{n}\left( 1-\delta \right) \right]  \notag \\
&&\left[ 1-\beta _{n}^{1}\left( 1-\delta \right) \right] \cdots \left[
1-\beta _{n}^{k-2}\left( 1-\delta \right) \right] \beta _{n}^{k-1}\varphi
\left( \left\Vert x_{n}-Tx_{n}\right\Vert \right)  \notag \\
&&+\left[ 1-\alpha _{n}\left( 1-\delta \right) \right]  \notag \\
&&\left[ 1-\beta _{n}^{1}\left( 1-\delta \right) \right] \cdots \left[
1-\beta _{n}^{k-3}\left( 1-\delta \right) \right] \beta _{n}^{k-2}\varphi
\left( \left\Vert y_{n}^{k-1}-Ty_{n}^{k-1}\right\Vert \right)  \notag \\
&&+\cdots +\left[ 1-\alpha _{n}\left( 1-\delta \right) \right] \beta
_{n}^{1}\varphi \left( \left\Vert y_{n}^{2}-Ty_{n}^{2}\right\Vert \right)
+\alpha _{n}\varphi \left( \left\Vert y_{n}^{1}-Ty_{n}^{1}\right\Vert \right)
\notag \\
&&+\left[ 1-\alpha _{n}\left( 1-\delta \right) \right] \left[ 1-\beta
_{n}^{1}\left( 1-\delta \right) \right] \cdots \left[ 1-\beta
_{n}^{k-2}\left( 1-\delta \right) \right] \beta _{n}^{k-1}\varepsilon  \notag
\\
&&+\cdots +\left[ 1-\alpha _{n}\left( 1-\delta \right) \right] \beta
_{n}^{1}\varepsilon +\alpha _{n}\varepsilon \text{.}  \label{eqn33}
\end{eqnarray}%
Since $\delta \in \left[ 0,1\right) $ and $\left\{ \alpha _{n}\right\}
_{n=0}^{\infty }$,$\left\{ \beta _{n}^{i}\right\} _{n=0}^{\infty }\subset %
\left[ 0,1\right) $ for $i=\overline{1,k-1}$, we have%
\begin{equation}
\left[ 1-\alpha _{n}\left( 1-\delta \right) \right] \left[ 1-\beta
_{n}^{1}\left( 1-\delta \right) \right] \cdots \left[ 1-\beta _{n}^{i}\left(
1-\delta \right) \right] \leq \left[ 1-\alpha _{n}\left( 1-\delta \right) %
\right] \text{.}  \label{eqn34}
\end{equation}%
Using inequality (2.10) and assumption (i) in (2.9) it follows%
\begin{eqnarray}
\left\Vert x_{n+1}-u_{n+1}\right\Vert &\leq &\left[ 1-\alpha _{n}\left(
1-\delta \right) \right] \left\Vert x_{n}-u_{n}\right\Vert  \notag \\
&&+\alpha _{n}\varphi \left( \left\Vert x_{n}-Tx_{n}\right\Vert \right)
+\alpha _{n}\varphi \left( \left\Vert y_{n}^{k-1}-Ty_{n}^{k-1}\right\Vert
\right)  \notag \\
&&+\cdots +\alpha _{n}\varphi \left( \left\Vert
y_{n}^{2}-Ty_{n}^{2}\right\Vert \right) +\alpha _{n}\varphi \left(
\left\Vert y_{n}^{1}-Ty_{n}^{1}\right\Vert \right)  \notag \\
&&+\alpha _{n}\varepsilon +\alpha _{n}\varepsilon +\cdots +\alpha
_{n}\varepsilon +\alpha _{n}\varepsilon  \notag \\
&=&\left[ 1-\alpha _{n}\left( 1-\delta \right) \right] \left\Vert
x_{n}-u_{n}\right\Vert  \notag \\
&&+\alpha _{n}\left( 1-\delta \right) \left\{ \frac{\varphi \left(
\left\Vert x_{n}-Tx_{n}\right\Vert \right) +\varphi \left( \left\Vert
y_{n}^{k-1}-Ty_{n}^{k-1}\right\Vert \right) }{1-\delta }\right.  \notag \\
&&\left. +\cdots +\frac{\varphi \left( \left\Vert
y_{n}^{1}-Ty_{n}^{1}\right\Vert \right) +k\varepsilon }{1-\delta }\right\} 
\text{.}  \label{eqn35}
\end{eqnarray}%
Define%
\begin{eqnarray*}
a_{n} &:&=\left\Vert x_{n}-u_{n}\right\Vert \text{,} \\
\mu _{n} &:&=\alpha _{n}\left( 1-\delta \right) \in \left( 0.1\right) \text{,%
} \\
\eta _{n} &:&=\left\{ \frac{\varphi \left( \left\Vert
x_{n}-Tx_{n}\right\Vert \right) +\varphi \left( \left\Vert
y_{n}^{k-1}-Ty_{n}^{k-1}\right\Vert \right) }{1-\delta }\right. \\
&&\left. +\cdots +\frac{\varphi \left( \left\Vert
y_{n}^{1}-Ty_{n}^{1}\right\Vert \right) +k\varepsilon }{1-\delta }\right\} 
\text{.}
\end{eqnarray*}%
From Theorem 1 it follows that $\lim_{n\rightarrow \infty }\left\Vert
x_{n}-p\right\Vert =0$. Since $T$ satisfies condition (1.4) and $Tp=p\in
F_{T}$,%
\begin{eqnarray}
0 &\leq &\left\Vert x_{n}-Tx_{n}\right\Vert  \notag \\
&\leq &\left\Vert x_{n}-p\right\Vert +\left\Vert Tp-Tx_{n}\right\Vert  \notag
\\
&\leq &\left\Vert x_{n}-p\right\Vert +\delta \left\Vert p-x_{n}\right\Vert
+\varphi \left( \left\Vert p-Tp\right\Vert \right)  \notag \\
&=&\left( 1+\delta \right) \left\Vert x_{n}-p\right\Vert \rightarrow 0\text{
as }n\rightarrow \infty \text{.}  \label{eqn36}
\end{eqnarray}%
Since $\beta _{n}^{i}\in \left[ 0,1\right) $, $\forall n\in 
\mathbb{N}
$, $i=\overline{1,k-1}$ and using (1.4) and (1.10) we have%
\begin{eqnarray}
0 &\leq &\left\Vert y_{n}^{1}-Ty_{n}^{1}\right\Vert =\left\Vert
y_{n}^{1}-p+p-Ty_{n}^{1}\right\Vert  \notag \\
&\leq &\left\Vert y_{n}^{1}-p\right\Vert +\left\Vert Tp-Ty_{n}^{1}\right\Vert
\notag \\
&\leq &\left\Vert y_{n}^{1}-p\right\Vert +\delta \left\Vert
p-y_{n}^{1}\right\Vert +\varphi \left( \left\Vert p-Tp\right\Vert \right) 
\notag \\
&=&\left( 1+\delta \right) \left\Vert y_{n}^{1}-p\right\Vert  \notag \\
&=&\left( 1+\delta \right) \left\Vert \left( 1-\beta _{n}^{1}\right)
y_{n}^{2}+\beta _{n}^{1}Ty_{n}^{2}-p\left( 1-\beta _{n}^{1}+\beta
_{n}^{1}\right) \right\Vert  \notag \\
&\leq &\left( 1+\delta \right) \left\{ \left( 1-\beta _{n}^{1}\right)
\left\Vert y_{n}^{2}-p\right\Vert +\beta _{n}^{1}\left\Vert
Ty_{n}^{2}-Tp\right\Vert \right\}  \notag \\
&\leq &\left( 1+\delta \right) \left\{ \left( 1-\beta _{n}^{1}\right)
\left\Vert y_{n}^{2}-p\right\Vert +\beta _{n}^{1}\delta \left\Vert
y_{n}^{2}-p\right\Vert \right\}  \notag \\
&=&\left( 1+\delta \right) \left[ 1-\beta _{n}^{1}\left( 1-\delta \right) %
\right] \left\Vert y_{n}^{2}-p\right\Vert  \notag \\
&\leq &\cdots  \notag \\
&\leq &\left( 1+\delta \right) \left[ 1-\beta _{n}^{1}\left( 1-\delta
\right) \right] \cdots \left[ 1-\beta _{n}^{k-2}\left( 1-\delta \right) %
\right] \left\Vert y_{n}^{k-1}-p\right\Vert  \notag \\
&\leq &\left( 1+\delta \right) \left[ 1-\beta _{n}^{1}\left( 1-\delta
\right) \right] \cdots \left[ 1-\beta _{n}^{k-1}\left( 1-\delta \right) %
\right] \left\Vert x_{n}-p\right\Vert  \notag \\
&\leq &\left( 1+\delta \right) \left\Vert x_{n}-p\right\Vert \rightarrow 0%
\text{ as }n\rightarrow \infty \text{.}  \label{eqn37}
\end{eqnarray}%
It is easy to see from (2.13) that this result is also valid for $\left\Vert
Ty_{n}^{2}-y_{n}^{2}\right\Vert ,\ldots ,\left\Vert
Ty_{n}^{k-1}-y_{n}^{k-1}\right\Vert $.

Since $\varphi $ is continuous, we have%
\begin{eqnarray}
&&\lim_{n\rightarrow \infty }\varphi \left( \left\Vert
x_{n}-Tx_{n}\right\Vert \right)  \notag \\
&=&\lim_{n\rightarrow \infty }\varphi \left( \left\Vert
y_{n}^{1}-Ty_{n}^{1}\right\Vert \right) =\cdots =\lim_{n\rightarrow \infty
}\varphi \left( \left\Vert y_{n}^{k-1}-Ty_{n}^{k-1}\right\Vert \right) =0%
\text{.}  \label{eqn38}
\end{eqnarray}%
Hence an application of Lemma 1 to (2.11) leads to%
\begin{equation}
\left\Vert p-q\right\Vert \leq \frac{k\varepsilon }{1-\delta }\text{.}
\label{eqn39}
\end{equation}
\end{proof}

As shown by Hussain et al. (\cite{Hussain}, Theorem 8), in an arbitrary
Banach space $X$, the S-iteration $\left\{ x_{n}\right\} _{n=0}^{\infty }$
given by (1.6) converges to the fixed point of $T$, where $T:E\rightarrow E$
is a mapping satisfying condition (1.3).

\begin{theorem}
Let $T:E\rightarrow E$ be a map satisfying $\left( 1.4\right) $ with $%
F_{T}\neq \emptyset $ and $\left\{ x_{n}\right\} _{n=0}^{\infty }$ be
defined by $(1.6)$ with real sequences $\left\{ \beta _{n}\right\}
_{n=0}^{\infty }$, $\left\{ \alpha _{n}\right\} _{n=0}^{\infty }\subset %
\left[ 0,1\right) $ satisfying $\sum\limits_{n=0}^{\infty }\alpha
_{n}=\infty $. Then the sequence $\left\{ x_{n}\right\} _{n=0}^{\infty }$
converges to the unique fixed point of $T$.
\end{theorem}

\begin{proof}
The argument is similar to the proof of Theorem 8 of \cite{Hussain}, and is
thus omitted.
\end{proof}

We now prove result on data dependence for the S-iterative procedure by
utilizing Theorem 3.

\begin{theorem}
Let $T$, $\widetilde{T}$ be two operators as in Theorem 2. Let $\left\{
x_{n}\right\} _{n=0}^{\infty }$, $\left\{ u_{n}\right\} _{n=0}^{\infty }$ be
S-iterations defined by (1.6) and with real sequences $\left\{ \beta
_{n}\right\} _{n=0}^{\infty }$, $\left\{ \alpha _{n}\right\} _{n=0}^{\infty
}\subset \left[ 0,1\right) $ satisfying (i) $\frac{1}{2}\leq \alpha _{n}$, $%
\forall n\in 
\mathbb{N}
$, and (ii) $\sum\limits_{n=0}^{\infty }\alpha _{n}=\infty $. If $p=Tp$\ and 
$q=\widetilde{T}q$, then we have%
\begin{equation*}
\left\Vert p-q\right\Vert \leq \frac{3\varepsilon }{1-\delta }\text{.}
\end{equation*}
\end{theorem}

\begin{proof}
For a given $x_{0}\in C$ and $u_{0}\in C$ we consider following iteration
for $T$ and $\widetilde{T}$:%
\begin{equation}
\left\{ 
\begin{array}{c}
x_{0}\in E\text{, \ \ \ \ \ \ \ \ \ \ \ \ \ \ \ \ \ \ \ \ \ \ \ \ \ \ \ \ \
\ \ \ \ \ \ \ \ \ \ \ \ } \\ 
x_{n+1}=\left( 1-\alpha _{n}\right) Tx_{n}+\alpha _{n}Ty_{n}\text{, \ \ \ \
\ \ \ \ \ \ \ \ } \\ 
y_{n}=\left( 1-\beta _{n}\right) x_{n}+\beta _{n}Tx_{n}\text{,}~n\in 
\mathbb{N}%
\end{array}%
\right.  \label{eqn40}
\end{equation}%
and%
\begin{equation}
\left\{ 
\begin{array}{c}
x_{0}\in E\text{, \ \ \ \ \ \ \ \ \ \ \ \ \ \ \ \ \ \ \ \ \ \ \ \ \ \ \ \ \
\ \ \ \ \ \ \ \ \ \ \ \ } \\ 
u_{n+1}=\left( 1-\alpha _{n}\right) \widetilde{T}u_{n}+\alpha _{n}\widetilde{%
T}v_{n}\text{, \ \ \ \ \ \ \ \ \ \ \ \ } \\ 
v_{n}=\left( 1-\beta _{n}\right) u_{n}+\beta _{n}\widetilde{T}u_{n}\text{,}%
~n\in 
\mathbb{N}
\text{.}%
\end{array}%
\right.  \label{eqn41}
\end{equation}%
Using (1.4), (2.16) and (2.17), we obtain the following estimates%
\begin{eqnarray}
\left\Vert x_{n+1}-u_{n+1}\right\Vert &=&\left\Vert \left( 1-\alpha
_{n}\right) \left( Tx_{n}-\widetilde{T}u_{n}\right) +\alpha _{n}\left(
Ty_{n}-\widetilde{T}v_{n}\right) \right\Vert  \notag \\
&\leq &\left( 1-\alpha _{n}\right) \left\Vert Tx_{n}-\widetilde{T}%
u_{n}\right\Vert +\alpha _{n}\left\Vert Ty_{n}-\widetilde{T}v_{n}\right\Vert
\notag \\
&=&\left( 1-\alpha _{n}\right) \left\Vert Tx_{n}-Tu_{n}+Tu_{n}-\widetilde{T}%
u_{n}\right\Vert  \notag \\
&&+\alpha _{n}\left\Vert Ty_{n}-Tv_{n}+Tv_{n}-\widetilde{T}v_{n}\right\Vert 
\notag \\
&\leq &\left( 1-\alpha _{n}\right) \left\{ \left\Vert
Tx_{n}-Tu_{n}\right\Vert +\left\Vert Tu_{n}-\widetilde{T}u_{n}\right\Vert
\right\}  \notag \\
&&+\alpha _{n}\left\{ \left\Vert Ty_{n}-Tv_{n}\right\Vert +\left\Vert Tv_{n}-%
\widetilde{T}v_{n}\right\Vert \right\}  \notag \\
&\leq &\left( 1-\alpha _{n}\right) \left\{ \delta \left\Vert
x_{n}-u_{n}\right\Vert +\varphi \left( \left\Vert x_{n}-Tx_{n}\right\Vert
\right) +\varepsilon \right\}  \notag \\
&&+\alpha _{n}\left\{ \delta \left\Vert y_{n}-v_{n}\right\Vert +\varphi
\left( \left\Vert y_{n}-Ty_{n}\right\Vert \right) +\varepsilon \right\} 
\text{,}  \label{eqn42}
\end{eqnarray}%
\begin{eqnarray}
\left\Vert y_{n}-v_{n}\right\Vert &=&\left\Vert \left( 1-\beta _{n}\right)
\left( x_{n}-u_{n}\right) +\beta _{n}\left( Tx_{n}-\widetilde{T}u_{n}\right)
\right\Vert  \notag \\
&\leq &\left( 1-\beta _{n}\right) \left\Vert x_{n}-u_{n}\right\Vert +\beta
_{n}\left\Vert Tx_{n}-\widetilde{T}u_{n}\right\Vert  \notag \\
&=&\left( 1-\beta _{n}\right) \left\Vert x_{n}-u_{n}\right\Vert +\beta
_{n}\left\Vert Tx_{n}-Tu_{n}+Tu_{n}-\widetilde{T}u_{n}\right\Vert  \notag \\
&\leq &\left( 1-\beta _{n}\right) \left\Vert x_{n}-u_{n}\right\Vert +\beta
_{n}\left\{ \left\Vert Tx_{n}-Tu_{n}\right\Vert +\left\Vert Tu_{n}-%
\widetilde{T}u_{n}\right\Vert \right\}  \notag \\
&\leq &\left( 1-\beta _{n}\right) \left\Vert x_{n}-u_{n}\right\Vert +\beta
_{n}\left\{ \delta \left\Vert x_{n}-u_{n}\right\Vert +\varphi \left(
\left\Vert x_{n}-Tx_{n}\right\Vert \right) +\varepsilon \right\}  \notag \\
&=&\left[ 1-\beta _{n}\left( 1-\delta \right) \right] \left\Vert
x_{n}-u_{n}\right\Vert +\beta _{n}\varphi \left( \left\Vert
x_{n}-Tx_{n}\right\Vert \right) +\beta _{n}\varepsilon \text{.}
\label{eqn43}
\end{eqnarray}%
Combining (2.18) and (2.19),%
\begin{eqnarray}
\left\Vert x_{n+1}-u_{n+1}\right\Vert &\leq &\left\{ \left( 1-\alpha
_{n}\right) \delta +\alpha _{n}\delta \left[ 1-\beta _{n}\left( 1-\delta
\right) \right] \right\} \left\Vert x_{n}-u_{n}\right\Vert  \notag \\
&&+\left\{ 1-\alpha _{n}+\alpha _{n}\delta \beta _{n}\right\} \varphi \left(
\left\Vert x_{n}-Tx_{n}\right\Vert \right) +\alpha _{n}\varphi \left(
\left\Vert y_{n}-Ty_{n}\right\Vert \right)  \notag \\
&&+\alpha _{n}\delta \beta _{n}\varepsilon +\left( 1-\alpha _{n}\right)
\varepsilon +\alpha _{n}\varepsilon \text{.}  \label{eqn44}
\end{eqnarray}%
For $\left\{ \alpha _{n}\right\} _{n=0}^{\infty }~$, $\left\{ \beta
_{n}\right\} _{n=0}^{\infty }\subset \left[ 0,1\right) $ and $\delta \in %
\left[ 0,1\right) $%
\begin{equation}
\left( 1-\alpha _{n}\right) \delta <1-\alpha _{n}\text{, }1-\beta _{n}\left(
1-\delta \right) <1\text{, }\alpha _{n}\delta \beta _{n}<\alpha _{n}\text{.}
\label{eqn45}
\end{equation}%
It follows from assumption (i) that%
\begin{equation}
1-\alpha _{n}<\alpha _{n}\text{,}\forall n\in 
\mathbb{N}
\text{.}  \label{eqn46}
\end{equation}%
Therefore, combining (2.22) and (2.21) to (2.20) gives%
\begin{eqnarray}
\left\Vert x_{n+1}-u_{n+1}\right\Vert &\leq &\left[ 1-\alpha _{n}\left(
1-\delta \right) \right] \left\Vert x_{n}-u_{n}\right\Vert  \notag \\
&&+2\alpha _{n}\varphi \left( \left\Vert x_{n}-Tx_{n}\right\Vert \right)
+\alpha _{n}\varphi \left( \left\Vert y_{n}-Ty_{n}\right\Vert \right)  \notag
\\
&&+\alpha _{n}\varepsilon +\alpha _{n}\varepsilon +\alpha _{n}\varepsilon 
\text{,}  \label{eqn47}
\end{eqnarray}%
or, equivalently,%
\begin{eqnarray}
\left\Vert x_{n+1}-u_{n+1}\right\Vert &\leq &\left[ 1-\alpha _{n}\left(
1-\delta \right) \right] \left\Vert x_{n}-u_{n}\right\Vert  \notag \\
&&+\alpha _{n}\left( 1-\delta \right) \frac{\left\{ 2\varphi \left(
\left\Vert x_{n}-Tx_{n}\right\Vert \right) +\varphi \left( \left\Vert
y_{n}-Ty_{n}\right\Vert \right) +3\varepsilon \right\} }{1-\delta }\text{.}
\label{eqn48}
\end{eqnarray}%
Now define%
\begin{eqnarray*}
a_{n} &:&=\left\Vert x_{n}-u_{n}\right\Vert \text{,} \\
\eta _{n} &:&=\alpha _{n}\left( 1-\delta \right) \in \left( 0,1\right) \\
\rho _{n} &:&=\frac{2\varphi \left( \left\Vert x_{n}-Tx_{n}\right\Vert
\right) +\varphi \left( \left\Vert y_{n}-Ty_{n}\right\Vert \right)
+3\varepsilon }{1-\delta }\text{.}
\end{eqnarray*}%
From Theorem 3, we have $\lim_{n\rightarrow \infty }\left\Vert
x_{n}-p\right\Vert =0$. Since $T$ satisfies condition (1.4), and $Tp=p\in
F_{T}$, using an argument similar to that in the proof of Theorem 2%
\begin{equation}
\lim_{n\rightarrow \infty }\left\Vert x_{n}-Tx_{n}\right\Vert
=\lim_{n\rightarrow \infty }\left\Vert y_{n}-Ty_{n}\right\Vert =0\text{.}
\label{eqn49}
\end{equation}%
Using the fact that $\varphi $ is continuous we have 
\begin{equation}
\lim_{n\rightarrow \infty }\varphi \left( \left\Vert x_{n}-Tx_{n}\right\Vert
\right) =\lim_{n\rightarrow \infty }\varphi \left( \left\Vert
y_{n}-Ty_{n}\right\Vert \right) =0\text{.}  \label{eqn50}
\end{equation}%
An application of Lemma 1 to (2.24) leads to%
\begin{equation}
\left\Vert p-q\right\Vert \leq \frac{3\varepsilon }{1-\delta }\text{.}
\label{eqn51}
\end{equation}
\end{proof}

\section{Conclusion}

Since the iterative schemes (1.7) and (1.8) are special cases the iterative
process (1.10), Theorem 1 generalizes Theorem 2.1 of \cite{CR}, and Theorem
2.1 of \cite{Isa}. By taking $k=3$ and $k=2$ in Theorem 2, data dependence
results for the iterative schemes (1.7) and (1.8) can be easily obtained.
For $k=3$, Theorem 2 reduces to Theorem 3.2 of \cite{DataSP}. Since
condition (1.4) is more general than condition (1.3), Theorem 3 generalizes
Theorem 8 of \cite{Hussain}.

\begin{acknowledgement}
This work is supported by Y\i ld\i z Technical University Scientific
Research Projects Coordination Unit under project number BAPK
2012-07-03-DOP02.
\end{acknowledgement}


\begin{thebibliography}{99}
\bibitem{Rafiq} A. Rafiq, \textit{"On the convergence of the three step
iteration process in the class of quasi-contractive operators"}, Acta
Mathematica Academiae Paedagogicae Nyiregyhaziensis, \textbf{22}(2006),
305-309.

\bibitem{Rhoades} B.E. Rhoades, \textit{"A comparison of various definitions
of contractive mappings"}, Transactions of the American Mathematical
Society, \textbf{226}(1977), 257-290.

\bibitem{RS7} B.E. Rhoades, S.M. \c{S}oltuz, \textit{"The equivalence
between Mann-Ishikawa iterations and multiste\textit{p} iteration"},
Nonlinear Analysis, \textbf{58}(2004), 219-228.

\bibitem{XuNoor} B. Xu, M.A. Noor, \textit{"Ishikawa and Mann iteration
process with errors for nonlinear strongly accretive operator equations"},
J. Math. Anal. Appl., \textbf{224}(1998), 91-101.

\bibitem{Chifu} C. Chifu, G. Petru\c{s}el, \textit{"Existence and Data
Dependence of Fixed Points and Strict Fixed Points for Contractive-Type
Multivalued Operators"}, Fixed Point Theory and Applications, \textbf{2007}%
(2007), Article ID 034248, 8 pages.

\bibitem{Chidume} C. E. Chidume, C. O. Chidume, \textit{"Convergence theorem
for fixed points of uniformly continuous generalized phihemicontractive
mappings"}, J. Math. Anal.Appl., \textbf{303}(2005), 545-554.

\bibitem{Chidume1} C.E. Chidume, C.O. Chidume, \textit{"Iterative
approximation of fixed points of nonexpansive mappings"}, J. Math. Anal.
Appl., \textbf{318}(2006), no. 1, 288-295.

\bibitem{Imoru} C.O. Imoru, M.O. Olantiwo, \textit{"On the stability of
Picard and Mann iteration processes"}, Carpathian Journal of Mathematics, 
\textbf{19}(2003), no.2, 155-160.

\bibitem{Sahu} D.R. Sahu, \textit{"Applications of the S-iteration process
to constrained minimization problems and split feasibility problems"}, Fixed
Point Theory, \textbf{12}(2011), no. 1, 187-204.

\bibitem{Das} G. Das, J.P. Debata, \textit{"Fixed points of
Quasi-nonexpansive mappings"}, Indian J. Pure, Appl. Math., \textbf{17}%
(1986), 1263-1269.

\bibitem{Akewe} H. Akewe, "Strong convergence and stabilitiy of
Jungck-multistep-SP iteration for generalized contractive-like inequality
operators", Advances in Natural Science, \textbf{5}(2012), no. 3, 21-27.

\bibitem{Rus1} I.A. Rus, S. Muresan, \textit{"Data dependence of the fixed
points set of weakly Picard operators"}, Stud. Univ. Babes-Bolyai 43 (1998)
79-83.

\bibitem{Rus2} I.A. Rus, A. Petru\c{s}el, A. S\^{\i}ntamarian, \textit{"Data
dependence of the fixed points set of multivalued weakly Picard operators"},
Stud. Univ. Babes-Bolyai, Math. 46 (2) (2001) 111-121.

\bibitem{Rus3} I.A. Rus, A. Petru\c{s}el, A. S\^{\i}ntamarian, \textit{"Data
dependence of the fixed point set of some multivalued weakly Picard
operators"}, Nonlinear Analysis: Theory, Methods \& Applications, 52 (2003),
1947--1959

\bibitem{Isa} \.{I}. Y\i ld\i r\i m, M. \"{O}zdemir, H. K\i z\i ltun\c{c}, 
\textit{"On the convergence of a new two-step iteration in the class of
quasi-contractive operators"}, Int. Journal of Math. Analysis, \textbf{3}%
(2009), no. 38, 1881-1892.

\bibitem{Markin} J.T. Markin, \textit{"Continuous dependence of fixed point
sets"}, Proc. AMS 38 (1973) 545-547.

\bibitem{Noor} M.A. Noor, \textit{"New approximation schemes for general
variational inequalities"}, Journal of Mathematical Analysis and
Applications, \textbf{251}(2000), no.1, 217-229.

\bibitem{Olantiwo} M. O. Olatinwo, \textit{"Some results on the continuous
dependence of the fixed points in normed linear space"}, Fixed Point Theory, 
\textbf{10}(2009), no. 1, 151-157.

\bibitem{Olantiwo1} M. O. Olatinwo, \textit{"On the continuous dependence of
the fixed points for }$\left( \varphi \text{,}\psi \right) $-\textit{%
contractive-type operators"}, Kragujevac Journal of Mathematics, \textbf{34}%
(2010), 91-102.

\bibitem{Osilike} M.O. Osilike, A. Udomene, \textit{"Short proofs of
stability results for fixed point iteration procedures for a class of
contractive-type mappings"}, Indian Journal of Pure and Applied Mathematics, 
\textbf{30}(1999), no.12, 1229-1234.

\bibitem{Hussain} N. Hussain, A. Rafiq, B. Damjanovi\'{c}, R. Lazovi\'{c}, 
\textit{"On rate of convergence of various iterative schemes"}, Fixed Point
Theory and Applications, \textbf{2011}(2011), doi:10.1186/1687-1812-2011-45.

\bibitem{Espinola} R. Esp\'{\i}nola, A. Petru\c{s}el, \textit{"Existence and
data dependence of fixed points for multivalued operators on gauge spaces"},
J. Math. Anal. Appl., 309 (2005), 420--432.

\bibitem{Glowinski} R. Glowinski, P. Le Tallec,\textit{\ "Augmented
Langrangian and operator splitting methods in nonlinear mechanics"}, SIAM,
Philadelphia, 1989.

\bibitem{CR} R. Chugh, V. Kumar, \textit{"Strong convergence of SP iterative
scheme for quasi-contractive operators in Banach spaces"}, International
Journal of Computer Applications, \textbf{31}(2011), no.5.

\bibitem{DataSP} R. Chugh, V. Kumar, \textit{"Data dependence of Noor and SP
iterative schemes when dealing with quasi-contractive operators"},
International Journal of Computer Applications, \textbf{31}(2011), no.5.

\bibitem{Agarwal} R.P. Agarwal, D. O'Regan, D.R. Sahu, \textit{"Fixed point
theory for lipschitzian type-mappings with applications"}, Springer, ISBN:
978-0-387-75817-6, 2009.

\bibitem{Agarwall} R.P. Agarwal, D. O'Regan, D.R. Sahu, \textit{"Iterative
construction of fixed points of nearly asymptotically nonexpansive mappings"}%
, J. Nonlinear Convex Anal., \textbf{8}(2007), no.1, 61-79.

\bibitem{Ishikawa} S. Ishikawa, \textit{"Fixed points by a new iteration
method"}, Proc. Amer. Math. Soc., \textbf{44}(1974), 147-150.

\bibitem{Suantai} S. Suantai, \textit{"Weak and strong convergence criteria
of Noor iterations for asymptotically nonexpansive mappings"}, J. Math.
Anal. Appl., \textbf{311}(2005), no.2, 506-517.

\bibitem{DataM} S.M. \c{S}oltuz, \textit{"Data dependence for Mann iteration"%
}, Octogon Math. Magazine, \textbf{9}(2001), 825-828.

\bibitem{Data Is 1} S.M. \c{S}oltuz, \textit{"Data dependence for Ishikawa
iteration"}, Lecturas Mathematicas, \textbf{25}(2004), no. 2, 149-155.

\bibitem{Data Is 2} S.M. \c{S}oltuz, T. Grosan, \textit{"Data dependence for
Ishikawa iteration when dealing with contractive like operators"}, Fixed
Point Theory and Applications, \textbf{2008}(2008), Article ID 242916, 7
pages.

\bibitem{Thianwan} S. Thianwan, \textit{"Common fixed points of new
iterations for two asymptotically nonexpansive nonself mappings in a Banach
space"}, J. Comput. Appl. Math., (2008), doi: 10.1016/j.cam.2008.05.051.

\bibitem{Zamfirescu} T. Zamfirescu, \textit{"Fix point theorems in metric
spaces"}, Archiv der Mathematik, \textbf{23}(1972), no.1, 292-298.

\bibitem{Vasile} V. Berinde, \textit{"Iterative Approximation of Fixed
Points "}, Springer, Berlin (2007).

\bibitem{VBerinde} V. Berinde, \textit{"On the approximation of fixed points
of weak contractive mappings"}, Carpathian J. Math., \textbf{19}(2003), no.
1, 7-22.

\bibitem{Berinde} V. Berinde, \textit{"On the convergence of the Ishikawa
iteration in the class of quasi contractive operators"}, Acta Mathematica
Universitatis Comenianae, \textbf{73}(2004), no.1, 119-126.

\bibitem{SP} W. Phuengrattana, S. Suantai, \textit{"On the rate of
convergence of Mann, Ishikawa, Noor and SP iterations for continuous
functions on an arbitrary interval"}, Journal of Computational and Applied
Mathematics, \textbf{235}(2011), 3006-3014.

\bibitem{Mann} W.R. Mann, \textit{"Mean value methods in iterations"}, Proc.
Amer. Math. Soc., \textbf{4}(1953), 506-510.

\bibitem{Takahashi} W. Takahashi, \textit{"Iterative methods for
approximation of fixed points and their applications"}, Journal of the
Operations Research Society of Japan, \textbf{43}(2000), no.1, 87-108.
\end{thebibliography}
\end{document}